\documentclass[12pt,a4paper,final]{article}
\usepackage{latexsym,amsthm,amssymb,amscd,amsmath, graphics,amsfonts}
\usepackage{graphicx}
\usepackage{tocbibind}
\usepackage{makeidx}
\usepackage{array}
\makeindex
\theoremstyle{definition}
\newtheorem{definition}{Definition}[section]
\theoremstyle{theorem}
\newtheorem{theorem}[definition]{Theorem}
\newtheorem{lemma}[definition]{Lemma}
\newtheorem{proposition}[definition]{Proposition}
\newtheorem{corollary}[definition]{Corollary}
\newtheorem{remark}[definition]{Remark}
\newtheorem{conjecture}[definition]{Conjecture}
\theoremstyle{definition}
\newcommand{\subject}[1]{\begin{flushleft}\textbf{Mathematics  Subject Classification}: #1\end{flushleft}}
\newcommand{\keyword}[1]{\par\noindent \textbf{Keywords:} #1 }
\newcommand{\university}[1]{\\[3mm]{\small #1}}
\title{Nearly K\"ahler submanifolds of a space form}\author{Nikrooz Heidari\footnote{nikrooz.heidari@modars.ac.ir}
\university{Tarbait Modares University}\and Abbas
Heydari\footnote{aheydari@modares.ac.ir}\university{Tarbait Modares
University}}\date{\today}
\begin{document}
\maketitle
\begin{abstract}
In this article we study isometric immersions of nearly K\"ahler manifolds into a space form (specially Euclidean space) and
 show that every nearly K\"ahler submanifold of a space form has an
 umbilic foliation whose leafs are 6-dimensional nearly K\"ahler manifolds. Moreover using
 this foliation we show that there is no non-homogeneous 6-dimensional nearly K\"ahler submanifold of a space form.
 We prove some results towards a classification of nearly K\"ahler hypersurfaces in standard space forms.\\
\keyword{Nearly K\"ahler manifold, Isometric immersion, Totally
umbilic foliation} \subject{53B35,53C55}
\end{abstract}
\section{Introduction}
Nearly K\"ahler manifolds are one of the sixteen classes of almost
Hermitian manifolds in the Gray-Hervella classification
\cite{Her}.When Gray was working in weak holonomy he faced these
manifolds \cite{Gray1}, whose Riemannian curvature operators satisfy
certain identities. These identities are only slightly more
complicated than and resembling the corresponding formula for
Riemannian curvature operator of K\"ahler manifolds. Gray called
them nearly K\"ahler manifolds and he was able to show that many
results on the topology and geometry of K\"ahler manifolds
generalize to nearly K\"ahler Manifolds and discovered certain new
topological and geometric properties of these
manifolds\cite{Gray2,Gray4}.

In 2002 Nagy\cite{Nag1,Nag2} described nearly K\"ahler manifolds as
almost Hermitian manifolds whose canonical Hermitian connection has
parallel and totally skew-symmetric torsion and showed that any
complete strict nearly K\"ahler manifold is finitely covered by a
product of homogeneous 3-symmetric manifolds, twistor spaces over
quaternionic K\"ahler manifolds
with their canonical nearly K\"ahler structure and 6-dimensional strict nearly K\"ahler manifolds.

6-dimensional nearly K\"ahler manifolds appearing in Nagy
decomposition are Einstein of positive scalar curvature. In this
dimension such a structure is characterized by the existence of some
(at least locally) real Killing spinor\cite{Gru}. Combining this
with the fact that the first Chern class vanishes one could observe
that non-K\"ahler nearly K\"ahler 6-dimensional manifolds solve
most of the type II string equations\cite{Fri}. Other source of
interest for nearly K\"ahler structure in six dimension
is provided by their relation with geometries with torsion, $G_{2}$-holonomy and super-symmetric models.\\
Up to now, the only known examples of complete, 6-dimensional,
strict nearly K\"ahler manifolds are the 6-dimensional 3-symmetric
spaces endowed with their natural nearly K\"ahler structure, namely
\begin{align*}
S^{6}&=\frac{G_2}{Su(3)},\qquad S^{3}\times S^{3}=\frac{Su(2)\times
Su(2)}{<1>},\\
\qquad \mathbb{C}P^{3}&=\frac{Sp(2)}
{Su(2)\cdot U(1)},\qquad
\mathbb{F}^{3}=\frac{Su(3)}{U(1)\times U(1)}
\end{align*}
Butruile\cite{But} showed that these are homogeneous, complete
6-dimensional nearly K\"ahler manifolds. The question naturally
arises that {\it are there any non-homogeneous complete nearly K\"ahler
manifolds in six dimension?} We have the following conjecture by
Butrulle.
\begin{conjecture}
Every compact nearly K\"ahler manifold is a 3-symmetric space.
\end{conjecture}
This conjecture is still open. By the work of Nagy\cite{Nag1,Nag2},
it may be separated into two sub-conjectures. The first one is a
similar conjecture on quaternionic K\"ahler manifolds and symmetric
spaces, where there are several reasons to believe that it is true
(it was solved by Poon and Salamon\cite{poon} in dimension 8 and
recently by Haydee and Herrera\cite{hay} in dimension 12). The
second may be formulated as: {\it The only compact, simply
connected, irreducible (with respect to the holonomy of the
intrinsic Hermitian connection), 6-dimensional, nearly K\"ahler
manifolds is the sphere $S^{6}$}. This somehow concerns the core of
the nearly K\"ahler geometry and might explain the rareness of such
manifolds or the difficulty to produce non-homogeneous examples.

In this paper, to further explore the latter conjecture, we study
isometric immersions $f:M^{2n}\longrightarrow\mathbb{Q}_{c}$ from
nearly K\"ahler Manifolds into a space form (specially an Euclidean
space) of constant curvature $c$. We show that every nearly K\"ahler submanifold of
a space form has a foliation such the leaves of this foliation are six dimension,
homogeneous nearly K\"ahler manifolds. As a result we
show that if $n=3$ there is no non-homogeneous nearly K\"ahler
submanifold of a space form and prove that the only 6-dimensional,
irreducible (with respect to the holonomy of the intrinsic
connection), complete(compact), simply connected nearly K\"ahler
submanifolds of a space form is $S^{6}$. This proves the second part
of the Batrulle conjecture in the situation that the nearly K\"ahler
manifolds are immersed in a space form.
\begin{theorem}\label{main theorem}
Let $f:M^{2n}\longrightarrow\mathbb{Q}^{2n+p}_{c}$ be an isometric immersion from a complete, simply connected strictly nearly K\"ahler
manifold into a space form of constant curvature $c$, then there
is an involute totally umbilic
foliation on $M$ such that the leaves of this
foliation are 6-nearly K\"ahler homogeneous manifolds. Moreover,
each leaf coincides with a 6-dimensional nearly K\"ahler factor
appearing in the Nagy decomposition.
\end{theorem}
\begin{corollary}\label{collorally}
Suppose that $f:M^{6}\longrightarrow\mathbb{Q}_c$ is an isometric
immersion from a 6-dimensional complete nearly K\"ahler manifold
into a space form, then for every $x\in M$ there is a subgroup of
local isometries of $M$  which acts transitively on an open
neighborhood of $x$. If $M^{6}$ is simply connected, then it is a
homogeneous manifold.
\end{corollary}
\begin{definition}
An isometric immersion $f:M\longrightarrow\bar{M}$ between
Riemannian manifolds called a full isometric immersion if there is
no totally geodesic submanifold $N$ of $\bar{M}$ of dimension
strictly less than $dim\bar{M}$ such that the restriction of $f$ to
$N$ is an isometric immersion.
\end{definition}
For example $M$ is in $\mathbb{R}^{n}$ with full isometric immersion
if and only if it is not contained (under $f$) in any affine
hypersurface of $\mathbb{R}^{n}$.
\begin{theorem}\label{6-dim theorem}
Let $f:M^{6}\longrightarrow\mathbb{R}^{n}$ be an isometric immersion
from a complete (compact) simply connected strictly 6-nearly
K\"ahler manifold into the Euclidean space. If $M$ is
irreducible with respect to the holonomy of the intrinsic  Hermitian
connection then $M$ is locally isometric with $S^{6}$ and if $f$ is
a full isometric immersion then $f$ is equivalent with natural
embedding of $S^6$ in $R^{7}$ (namely, there exist isometric map
$j:\mathbb{R}^{7}\longrightarrow\mathbb{R}^{7}$ and locally
isometric map $g:M^{6}\longrightarrow S^6$ such that $j\circ
f=i\circ g$ where $i$ is the standard embedding of $S^6$ as a
quadric hypersurface in $\mathbb{R}^{7}$).
\end{theorem}
The paper is organized as follows. In section 2 basic definitions and preliminaries are given. Main reference of this section is \cite{But}. In section 3 we introduce complex and invariant totally umbilic foliation and prove the main theorem and its corollaries. In section 4 we use this foliation to classify nearly K\"ahler hypersurfaces in space forms using the principal curvature.

\section{Preliminaries}

Nearly K\"ahler manifolds are almost Hermitian, i.e. $2n$-dimensional real manifolds with a $U(n)$-structure (a $U(n)$-reduction of the frame bundle) or equivalently, with a pair of tensors $(g,J)$ or $(g,\omega)$ where $g$ is a Riemannian metric, $J$ is an almost complex structure  compatible with $g$ in the sense that $g(JX,JY)=g(X,Y)$ for each $X, Y\in TM$, ($J$ is orthogonal with respect to $g$ point-wise) and $\omega$ is a differential 2-form, called the K\"ahler form, related to $g$ and $J$ by $\omega(X,Y)=g(JX,Y)$, for $X,Y\in TM$.\\
Associated with $g$ there is the well-known Levi-Civita connection $\nabla$ which is metric preserving and torsion-free. But nearly K\"ahler manifolds, as every almost Hermitian manifolds, have another natural connection $\bar{\nabla}$ called the intrinsic connection or the canonical Hermitian connection, which shall be of considerable importance in the sequel. Let $\mathfrak{so}(M)$ be the bundle of skew-symmetric endomorphisms of the tangent space (the adjoint bundle of the metric structure). The set of metric connections of $(M,g)$ is an affine space $\mathcal{SO}$  modelled on the space of sections of $T^{*}M\otimes\mathfrak{so}(M)$.The set $\mathcal{U}$ of Hermitian connections (connections which preserve both the metric and the almost complex structure or the K\"ahler form) is an affine subspace of $\mathcal{SO}$ with vector space $\Gamma(T^{*}M\otimes\mathfrak{u}(M))$, where $\Gamma(T^{*}M\otimes\mathfrak{u}(M))$ is the subbundle of $\mathfrak{so}(M)$ formed by the endomorphisms which commute with $J$ (in other words, the adjoint bundle of the $U(n)$-structure). We denote the orthogonal complement of $\mathfrak{u}(M)$ in $\mathfrak{so}(M)$ by $\mathfrak{u}(M)^{\perp}$, it is identified with the bundle of skew-symmetric endomorphisms of $TM$, anti-commuting with $J$.
\begin{definition}
 The canonical Hermitian connection $\bar{\nabla}$ is the orthogonal projection of  $\nabla\in\mathcal{SO}$ on $\mathcal{U}$. Equivalently, it is the unique Hermitian connection such that $\nabla -\bar{\nabla}$ is a 1-form with values in $\mathfrak{u}(M)^{\perp}$.
\end{definition}
The difference $\eta=\nabla -\bar{\nabla}$ is known explicitly, $\eta X=\frac{1}{2}J\circ(\nabla_{X}J)$  for all $X\in TM$, and measures the failure of the $U(n)$-structure to admit a torsion-free connection. It can be used to classify almost Hermitian manifolds. For example, K\"ahler manifolds are characterized by $\nabla$ being a Hermitian connection: $\bar{\nabla}=\nabla$.
\begin{definition}
Let M be an almost Hermitian manifold. The following conditions are equivalent and define a nearly K\"ahler manifold:\\
(1) the torsion of $\bar{\nabla}$ is totally skew-symmetric,\\
(2) $(\nabla_{X}J)X=0$  for all $X\in TM$,\\
(3) $\nabla_{X}\omega=\frac{1}{3}\mathfrak{i}_{X}d\omega$ for all $X\in TM$,\\
(4) $d\omega$ is of type $(0,3)+(3,0)$ and the Nijenhuis tensor $N$ is totally skew-symmetric.
\end{definition}
\begin{proposition}
For a nearly K\"ahler manifold, the torsion of the intrinsic
Hermitian connection is totally skew-symmetric and
parallel, that is $\bar{\nabla}\eta=0$. Moreover, this is equivalent to
$\bar{\nabla}\nabla\omega=0$ or $\bar{\nabla}d\omega=0$.
\end{proposition}
Now, suppose that the curvature of $\bar{\nabla}$ is also parallel: $\bar{\nabla}\bar{R}=0$. Then M is locally homogeneous or an Ambrose-Singer manifold.
\begin{lemma}\label{formula}
\cite{Gray2}
Assume that $(M,g,J)$ is a nearly K\"ahler manifold then\\
\begin{eqnarray*}
 &(\nabla_{X}J)Y+(\nabla_{Y}J)X=0\\
 &(\nabla_{JX}J)JY=(\nabla_{X}J)JY \\
 &J(\nabla_{X}J)Y=-(\nabla_{X}J)JY=-(\nabla_{JX}J)Y\\
 &g(\nabla_{X}Y,X)=g(\nabla_{X}JY, JX)\\
 &2g((\nabla_{W,X}^2J)Y,Z)=-\sigma_{X,Y,Z}g((\nabla_WJ)X,(\nabla_{Y}J)JZ)
\end{eqnarray*}
\end{lemma}
Gray used following formulas to find relation between torsion of  intrinsic Hermitian connection and Riemannian curvature on nearly K\"ahler manifolds \cite{Gray4}. These formulas resemble the corresponding formulas for K\"ahler manifolds.\\
\begin{eqnarray*}
<R_{X,Y}X,Y>-<R_{X,Y}JX,JY>=\|(\nabla_{X}J)Y\|^{2}\\
<R_{W,X}Y,Z>-<R_{W,X}JY,JZ>=<(\nabla_{W}J)X,(\nabla_{Y}J)Z>\\
<R_{W,X}Y,Z>=<R_{JW,JX}JY,JZ>\\
2g((\nabla_{W,X}^{2}J)Y,Z)=\sigma_{X,Y,Z}g(R_{WJX}Y,Z)\\
\end{eqnarray*}

It is easy to check that on every nearly K\"ahler manifold, the tensors $A(X,Y,Z)=<(\nabla_{X}J)Y,Z>$ and $B(X,Y,Z)=<(\nabla_{X}J)Y,JZ>$ are skew-symmetric and have type $(3,0)+(0,3)$ as (real) 3-forms.\\
We need the following classical relation between the covariant derivative of the almost complex structure $J$ and its Nijenhuis tensor $N$ which is proved by a straightforward computation using
\begin{align*}
4N(X,Y)=[X,Y]+J[JX,Y]+J[X,JY]-[JX,JY]
\end{align*}
and the anti-symmetry of the tensors A and B defined above.
\begin{lemma}\label{nijenhus}
 For every nearly K\"ahler manifold $(M,g,J)$ we have
 \begin{align*}
\forall X,Y\in\chi(M)\quad N(X,Y)=J(\nabla_{X}J)Y.
\end{align*}
\end{lemma}
In lower dimensions, the nearly K\"ahler manifolds are mainly determined. If $M$ is nearly K\"ahler with $dimM\leq4$, then M is K\"ahler. If $dimM=6$, then we have the following:
\begin{proposition}\cite{Gray2,Gray4,wat1}\label{6-dim}
Let $(M,g, J)$ be a 6-dimensional, strict, nearly K\"ahler manifold. Then\\
 (1) $\nabla J$ has constant type, that is
 \begin{align*}
\|(\nabla_{X}J)Y\|^{2}=\frac{S}{30}{\|X\|^{2}\|Y\|^{2}-g(X,Y)^{2}-g(JX,Y)^{2}}
\end{align*}
\quad for all vector fields $X$ and $Y$,\\
(2) the first Chern class of $(M, J)$ vanishes,\\
(3) $M$ is an Einstein manifold;
\begin{align*}
Ricc=\frac{S}{6}g, \qquad Ricc^{*}=\frac{S}{30}g.
\end{align*}
Moreover if the tensor $\nabla J$ has constant type $\alpha$ then $dimM=6$ and $\alpha=\frac{S}{30}$ where $S$ is scalar curvature.
\end{proposition}
The next lemma follows immediately.
\begin{lemma}
 For vector fields $W, X,Y$ and $Z$ we have
 \begin{eqnarray*}
g((\nabla_{W}J)X,(\nabla_{Y}J)Z)=\frac{S}{30}\{g(W,Y)g(X,Z)-g(W,Z)g(X,Y)\\
-g(W,JY)g(X,JZ)+g(W,JZ)g(X,JY)\},
\end{eqnarray*}
and
\begin{eqnarray*}
 g((\nabla_{W}\nabla_{Z}J)X,Y)=\frac{S}{30}\{g(W,Z)g(JX,Y)
 -g(W,X)g(JZ,Y)+g(W,Y)(JZ,X)\},
 \end{eqnarray*}
 also
 \begin{eqnarray*}
  \Sigma g(Je_{i},e_{j})R(e_{i},e_{j}X,Y)=-\frac{S}{15}g(Jx,Y),\\
\Sigma g((\nabla_{X}J)e_{i},e{j})R(e_{i},e_{j},Y,Z)=-\frac{S}{30}g((\nabla_{X}J)Y,Z),
  \end{eqnarray*}
where $\{e_{i}\}$ is a local orthonormal frame field on $M$.
\end{lemma}
\section{Complex and invariant totally umbilic foliation}
Let $f:M^{n}\longrightarrow\mathbb{Q}^{n+p}_{c}$ be an isometric immersion from a complete oriented Riemannian manifold $M$ into the space form $\mathbb{Q}$ of codimension $p$ and $\alpha$ be its second fundamental form. For smooth normal section $\eta\in\Gamma(TM)^{\perp}$ and $x\in M$  the totally umbilic distribution (with singularity) associated to $\eta$ is defined by
\begin{align*}
\Delta_{x}=\{X\in T_{x}M| \, \alpha(X,Y)=<X,Y>\eta\qquad\forall Y\in T_{x}M\}.
\end{align*}
This distribution is smooth because $\Delta_{x}=\bigcap Ker\alpha_{i}(x)$ where
\begin{align*}
\alpha_{i}=\alpha(\cdot, X_{i})-<\cdot, X_{i}>\eta
\end{align*}
is a smooth 1-form for the local frame fields $\{X_{i}\}$.
Therefore $\Delta$ is a smooth
distribution with singularity (may not be of constant dimension) because these 1-forms may be linearly dependent. Put $\nu(x)=dim\Delta_{x}$, then $\nu$ is semi-continuous and  there is exist an open $U$ such that $\nu$  is
constant on $U$. This distribution is involutive because by the Codazzi equation we have
\begin{align*}
(\nabla_{X}^{\perp}\alpha)(Y,Z)=(\nabla_{Y}^{\perp}\alpha)(X,Z),
\end{align*}
hence
\begin{align*}
\nabla_{X}^{\perp}\alpha(Y,Z)-\alpha(\nabla_{X}Y,Z)-\alpha(Y,\nabla_{X}Z)=\nabla_{Y}^{\perp}\alpha(X,Z)-\alpha(\nabla_{Y}X,Z)-\alpha(X,\nabla_{Y}Z),
\end{align*}
and if $X,Y\in\Delta_{x}$ then for all $Z\in \Delta_{x}^{\perp}$ we have
\begin{align*}
\alpha(\nabla_{X}Y,Z)+<\nabla_{X}Y,Z>\eta=\alpha(\nabla_{Y}X,Z)+<\nabla_{Y}X,Z>\eta,
\end{align*}
and for $Z\in \Delta_{x}$ there is nothing to prove.
On the other hand, $[X,Y]=\nabla_{X}Y-\nabla_{Y}X$,  therefore
\begin{align*}
\alpha([X,Y],Z)=<[X,Y],Z>\eta.
\end{align*}
Put
\begin{align*}
\Delta_{x}^{'}=\{X\in T_{x}M| \forall Y\in T_{x}M\quad\alpha(X,JY)+\alpha(JX,Y)=0\},
\end{align*}
where $J$ is the almost complex structure. Like the previous case, this distribution is smooth but may not be involutive. When $M$  is almost Hermitian, the complexification of totally umbilic distribution at each point is defined by
\begin{align*}
\Delta_{x}\cap J\Delta{x}=\Delta_{x}\cap\Delta_{x}^{'}.
\end{align*}
This distribution is smooth but may be not involutive. When $M$ is nearly K\"ahler with torsion $T$ of  intrinsic Hermitian connection, we define
\begin{align*}
\Delta_{x}^{''}=\{X\in T_{x}M|  \forall Y,Z\in T_{x}M\quad\alpha(T(X,Y),Z)+\alpha(T(X,Z),Y)=0\}.
\end{align*}
The corresponding distribution is smooth but not involutive.\\
\begin{definition}
Let $f:M^{2n}\longrightarrow\mathbb{Q}^{2n+p}$ be an isometric immersion from a complete nearly K\"ahler manifold $M$ into a space form $\mathbb{Q}$. We denote by $\alpha$ the second fundamental form of $f$ and by $T$ the torsion of intrinsic Hermitian connection on $M$ and define the complex and invariant totally umbilic distribution
at each point by
\begin{align*}
D_{x}^{\eta}=D_{x}=\Delta_{x}\cap\Delta_{x}^{'}\cap\Delta_{x}^{''},
\end{align*}
which can be easily seen to be equal to
\begin{align*}
D_{x}=\{X\in T_{x}M| \forall
Y\in T_{x}M\quad X\in\Delta_{x},JX\in\Delta_{x},T(X,Y)\in\Delta_{x}\}
\end{align*}
\end{definition}

\begin{lemma}\label{main lemma}
Let $f:M^{2n}\longrightarrow\mathbb{Q}^{2n+p}_{c}$ be an isometric immersion from a complete and strictly nearly K\"ahler manifold into a space form of curvature $c$ and codimension $p$. Then the complex and invariant totally umbilic distribution is smooth and involutive and it defines a foliation with singularity. The leaf of this foliation may not be complete even if $M$ is compete.
\end{lemma}
\begin{proof}
 By definition and smoothness of $\Delta_{x},\Delta_{x}^{'},\Delta_{x}^{''}$ we conclude that $D_{x}$ is smooth. Therefore $D_{x}$ defines a singular distribution on $M$. We show that this distribution is involutive. Let $X,Y\in D_{x}$ by Lemma \ref{nijenhus} and the fact that $D_{x}$ is invariant under $T,J$ in $\Delta$ we have $N(X,Y)\in D_{x}$ where $N$ is the Nijenhuis tensor of almost complex structure $J$. By definition of $N$ and since $\Delta_{x}$ is involutive we have
\begin{align*}
N(X,Y)-[X,Y]+[JX,JY]=J[JX,Y]+J[X,JY]\in\Delta_{x}.
\end{align*}
For each $Z\in \Delta_{x}^{\perp}$ we have $<J[X,Y],-J[JX,JY],Z>=0$, therefore
\begin{align*}
 \alpha(J[X,Y],Z)=\alpha(J[JX,JY],Z),
 \end{align*}
and a computation like that of \cite{Flo1,Flo3} and using fact that $N$ is of skew-symmetric implies that $J[X,Y]\in\Delta_{x}$.\\
To show that the distribution $D$ is involutive we show that for all $Z\in T_{x}M$, $T([X,Y],Z)\in\Delta_{x}$. For this purpose we define the following tensor
 \begin{align*}
&C_{X}:D_{x}^{\perp}\longrightarrow D_{x}^{\perp}\\
&C_{X}(Z)=T(X,Z)^{\perp}.
\end{align*}
\end{proof}
\begin{lemma}\label{tensor}
The tensor $C_{X}$ vanishes for all $X\in D_{x}$.
\end{lemma}
\begin{proof}
Let $Y\in D_{x}$ be such that $<X,Y>=1$ and $W\in \Delta_{x}^{\perp}$. Note that $T$ is $\bar{\nabla}$-parallel and  we have
\begin{align*}
&0=(\nabla_{X}T)(Y,Z)+T(X,T(Y,Z)-T(T(X,Y),Z)-T(Y,T(X,Z)),
\end{align*}
and since $X,Y\in D_{x}$, $<T(X,T(Y,Z),W>=<T(Y,T(X,Z),W>=0$, thus using Gray formula at the end of Lemma \ref{formula} we conclude that
\begin{align*}
<T(T(X,Y),Z),W>&=<T(X,Y),T(W,Z)>=<R_{X,JY}JW,Z>\\&+<R_{X,JY}W,JZ>,
\end{align*}
and if we use described Riemannian curvature by the second fundamental form
\begin{align*}
<R_{X,Y}Z,W>&=c\{<X,W><Y,Z>-<X,Z><Y,W>\}\\&+<\alpha(X,W),\alpha(Y,Z)>-<\alpha(X,Z),\alpha(Y,W)>
\end{align*}
we get $<T(T(X,Y),Z,W>=0$ (because  $X,Y\in D_{x}$). Therefore by Gray formula (8) we have
\begin{align*}
&0=\sigma_{Y,Z.W}<R_{X,JY}Z,W>=<\alpha(X,Y),\alpha(JZ,W)-\alpha(Z,JW)>\\
&\Longrightarrow <\eta,\alpha(JZ,W)-\alpha(Z,JW)>=0.
\end{align*}
\begin{remark}
Note that for $A\in T_xM, X\in D_{x}$, $Z\in D_{x}^{\perp}$ and $W\in\Delta_{x}^{\perp}$,
\begin{align*}
<(\nabla_{A}T)(X,Z),W>=<(\nabla_{A,X}^{2}J)JZ,W>,
\end{align*}
\end{remark}
Now to show that the tensor $C$ vanishes, let $A\in T_{x}M$, it is sufficient to show that for all $W\in \Delta_{x}^{\perp}$, $<T(T(X,Z),A),W>$ vanishes. But
\begin{align*}
<T(T(X,Z),A),W>&=<(\nabla_{A}T)(Z,X),W>+<T(X,T(A,Z)),W>\\&+<T(T(A,X),Z),W>
\end{align*}
and similar to the previous case, the last two terms on the right hand side are zero. Therefore
\begin{align*}
&<T(T(X,Z),A),W>=<(\nabla_{A}T)(X,Z),W>=\frac{1}{2}\sigma_{X,Z,W}<R_{A,JX}Z.W>\\
&=<\alpha(A,W),\alpha(JX,Z)>+<\alpha(A,Z),\alpha(JX,W)>\\&+\frac{1}{2}<\alpha(A,X),\alpha(JZ,W)-\alpha(JW,Z)>,
\end{align*}
when $JX\in D_{x}$ ,$Z\in D_{x}^{\perp}$ and $W\in\Delta_{x}^{\perp}$ the first two terms on the right hand side are also zero because $X\in D_{x}$ and $\alpha(X,A)=<X,A>\eta$ therefore the third term in the above equation vanishes. This means that $T(T(X,Z),A)\in\Delta_{x}$ for all $A\in T_{x}M$, therefore $T(X,Z)\in D_{x}$ and the proof is complete.
\end{proof}
\paragraph*{}
To proof that the distribution is involutive it is sufficient to show that for all  $X,Y\in D_{x}$ and $Z\in T_{x}M$, $T([X,Y],Z)\in\Delta_{x}$. The case $Z\in D_{x}$ is trivial therefore we may suppose that $Z\in D_{x}^{\perp}$. Also we may assume that $W\in\Delta_{x}^{\perp}$ then
\begin{align*}
&0=<\bar{\nabla}_{X}T(Y,Z),W>=<(\nabla_{X}T)(X,Y)>+<T(X,T(Y,Z)),W>,
\end{align*}
where the last term vanishes since $X\in D_{x}$ and $T(X,T(Y,Z))\in\Delta_{x}$ is equal to zero therefore
\begin{align*}
0=<\nabla_{X}T(Y,Z),W>-<T(\nabla_{X}Y,Z),W>-<T(Y,\nabla_{X}Z),W>
\end{align*}
the last term of this relation is also zero since the $Y\in D_{x}$ and
\begin{align*}
<T(\nabla_{X}Y,Z),W>&=<\nabla_{X}T(Y,Z),W>=<T(Y,Z),\nabla_{X}W>\\
&=-<T(Y,\nabla_{X}W),Z>=0.
\end{align*}
The last term is equal zero by Lemma \ref{tensor}  therefore
distribution is involutive. We know that the dimension of $\Delta$ on integral curves associated with the vector fields generated this distribution is constant \cite{Daj}. With a similar argument and using \cite{Ste} we observe that $D$ is constant on
integral curves associated with the generating vector fields of $D$
therefore by Sussman-Stefan theorem \cite{Mich,Sus}
$D$ is a foliation with singularity and has maximal locally
integral manifolds.
\begin{lemma}\label{com lemma}
Let $f:M^{2n}\longrightarrow\mathbb{Q}_{c}^{2n+p}$ be an isometric immersion from a complete and strictly nearly K\"ahler manifold into a space form with curvature $c$ and codimension $p$. Each leafs of complex and invariant totally umbilic foliation by  torsion of the intrinsic Hermitian connection is 6-dimensional homogeneous nearly K\"ahler manifolds.
\end{lemma}
\begin{proof}
Let $\theta(x)=dimD_{x}$ then by Lemma \ref{main lemma} and Sussman-Stefan theorem concluded that semi-continuous function $\theta(x)$ is locally constant. It means that for all $x\in M$ there exists an open
subset $U\subseteq M$ containing $x$ such that $\theta(x)$ on $U$ is
constant therefore $D_{x}$ on $U$ is integrable distribution. The maximal
integral manifold $N$ of this distribution  is nearly K\"aler manifold
with induced almost complex structure and inherited metric because $D$
is invariant under $J,T$. We assert that $N$
has constant type $c+<\eta,\eta>$. By the Gray formula 8 in
Lemma \ref{formula} for all $X,Y\in\chi(N)$
\begin{align*}
\|(\nabla_{X}J)Y\|^{2}=&-<R_{X,Y}X,Y>+<R_{X,Y}JX,JY>\\
&=<\alpha(X,Y),\alpha((Y,X)>-<\alpha(X,X),\alpha(Y,Y)>\\&-<\alpha(X,JY),\alpha(Y,JX)>+<\alpha(X,JX),\alpha(Y,JY)\\
&=\|\eta\|^{2}(-<X,Y>^{2}+<X,X><Y,Y>\\&+<X,JY><Y,JX>)\\
&=\|\eta\|^{2}(\|X\|^{2}\|Y\|^{2}-<X,Y>^{2}-<X,Y>^{2}),
\end{align*}
hence by Proposition \ref{6-dim} $N$ is 6-dimensional manifold and $\|\eta\|^{2}=\frac{S}{30}$ where $S$ be scalar curvature of $N$. Also because $N$ is Einstein $\eta$ is of constant length. By the definition of the tangent bundle $TN$ at each point we have $H=\eta$ where $H$ is mean curvature vector field of $N$ as a submanifold of $\mathbb{Q}$.\\
Noted that the tangent space of this leaf at each point $x$ is $D_{x}$ therefore
\begin{align*}
<R(X_{i},X_{j})X_{k},X_{l}>&=<\alpha(X_{i},X_{l}),\alpha(X_{j},X_{k})>-<\alpha(X_{i},X_{k}),\alpha(X_{j},X_{l})>\\
&=\|\eta\|^{2}\{\delta^{i}_{l}\delta^{j}_{k}-\delta^{i}_{k}\delta^{j}_{l}\}
\end{align*}
and a computation like that of \cite{Hei2} implies that $\bar{\nabla}\bar{R}=0$. Therefore by the above remark each leaf is an Ambrose-Singer manifold (locally homogeneous).
\end{proof}
\paragraph*{}
Since the  Nagy decomposition to homogeneous 3-symmetric nearly K\"ahler, twistor space on positive K\"ahler quaternion and 6-dimension nearly K\"ahler manifold factors are unique up to homothety the  proof of theorem\ref{main theorem} is complete.
\begin{proposition}
 In Nagy decomposition, we have 6-dimensional factor if and only if there exist $0 \neq\eta\in T^{\perp}_{f}M$ such that the $D^{\eta}$ is non-zero distribution.
\end{proposition}
\begin{proof}
By Theorem \ref{main theorem} the leaves of foliation $D$ are isometric to 6-dimensional factor in Nagy decomposition.

Conversely, let $N$ be the  6-dimensional nearly K\"ahler factor in Nagy decomposition with second fundamental form $\beta$ as a submanifold of $M$, so $\beta(X,JY)=J\beta(X,Y)$ for all $X,Y\in TN$. One can consider $N$ as a submanifold of $\mathbb{Q}^{2n+p}$  (by $N^{6}\longrightarrow M^{2n}\longrightarrow\mathbb{Q}^{2n+p}$). If $H$ is the mean curvature of $N$ in $\mathbb{Q}$ we choose $\eta$ parallel to  $H=\sum_{i=1}^{i=3}\beta(e_{i},e_{i})+\beta(Je_{i},Je_{i})$ such that $c+||\eta ||=\frac{S}{30}$, where $\{e_{i},Je_{i}\}$ is suitable Watanabe frame \cite{wat1} for $N$ and $S$ is the scalar curvature of $N$. Therefore $\eta\in T^{\perp}_{f}M$ and $T^{M}(X,Y)=T^{N}(X,Y)$ for all $X,Y\in TN$ where $T^{N}$ and $T^{M}$ are torsions of canonical Hermitian connection on $N$ and $M$ respectively.  By computation like as the proof of lemma \ref{com lemma} we have $\beta(X,Y)=g(X,Y)\eta$ for all $X,Y\in TN$. So $D^{\eta}$ is non-zero.
\end{proof}
\begin{proof}
(\noindent\textbf{corollary \ref{collorally}})
From the previous lemma,  if $M$ is 6-dimensional each leaf of complex and invariant totally umbilic foliation is an open subset of $M$ and because each leaf is locally homogeneous, $M$ is locally homogeneous and by Ambrose-Singer theorem if $M$ is simply connected then $M$ is homogeneous. Therefore there is no non-homogeneous complete simply connected and strict nearly K\"ahler submanifold of a space form, namely such a manifold must be one in the Butrulle classification of complete homogeneous 6-nearly K\"ahler manifolds.\end{proof}
\begin{proof}
(\noindent\textbf{theorem \ref{6-dim theorem}}) Since $M$ has irreducible
holonomy of  intrinsic Hermitian
connection, from the proof of Nagy decomposition each leaf of
complex and invariant totally umbilic foliation is complete . Therefore by
Corollary \ref{collorally} each leaf of this foliation
on $M$ is open, compact and thus is closed therefore as
$M$ is connected, the only leaf is itself $M$. But
$M$ is simply connected hence it must be one of the manifolds listed in
Butrulle theorem. Also from \cite{bel} we know that
$\mathbb{C}P^{3},\mathbb{F}^{3}$ with their standard nearly K\"ahler
structure coming from the twistor construction have reducible
holonomy of  intrinsic Hermitian connection and $S^{3}\times S^{3}$
is Riemannian reducible. Therefore the only manifold which is irreducible
with respect to the holonomy of the intrinsic Hermitian
connection must be $S^{6}$.
\begin{remark}
The only totally umbilic irreducible Euclidean submanifold is the sphere and the only sphere with strictly nearly K\"ahler structure is 6-sphere.
\end{remark}
If $\gamma$  is a normal vector field on $M$ which does not vanish everywhere,
for the complex and invariant totally umbilic foliation $D^{\gamma}$ we have $T_{x}M=D^{\gamma}_{x}=D^{\eta}_{x}$ and for $0\neq X\in T_{x}M$,  $\alpha(X,X)=<X,X>\eta_{x}=<X,X>\gamma_{x}$ thus $\eta_{x}=\gamma_{x}$ for all $x\in M$, therefore
\begin{align*}
\mathcal{N}^{1}_{x}:=spam\{\alpha(X,Y)|X,Y\in T_{x}M\}=<\eta_{x}>.
\end{align*}
In particular,
 $\mathcal{N}^{1}$ is an invariant distribution under parallel translation with respect to the induced connection on normal bundle $\nabla^{\perp}$, thus by reduction of codimension theorem \cite{Erb} there exist a (6+1)-dimensional totally geodesic submanifold $N$ of $\mathbb{R}^{n}$ such that $f$ is an isometric immersion of $M$ into $N$. This contradicts the fact that $f$ is full isometric immersion,  therefore $n$ must be 7. The rest part of theorem is already considered in \cite{Ber}.
 \end{proof}
\section{Classification of nearly K\"ahler hypersurfaces of a space form}
In this section using complex and invariant totally umbilic foliations we classify and describe nearly K\"ahler hypersurfaces of a space form based on their principal curvatures.
\paragraph*{}
In the previous section, we observed that there is no non-homogeneous complete nearly K\"ahler submanifold in a space form. If $dimM=6$ and $M$ is simply connected, there exist only four cases (up to homothety): $S^{6}$, $S^{3}\times S^{3}$, $\mathbb{C}P^{3}$ and $\mathbb{F}^{3}$.

Using Reyes Carrion structure theorem \cite{Rey} and Theorem1.1 of \cite{bel} which state that every complete Nearly K\"ahler, non-K\"ahler manifold, whose canonical connection has reduce holonomy is homothetic to $\mathbb{C}P^{3}$ or $\mathbb{F}^{3}$ with their standard nearly K\"ahler structure, we show that these type of manifolds
cannot be isometrically immersed hypersurfaces in a standard space form.
%
\begin{theorem}
If $f:M^{6}\longrightarrow\mathbb{Q}_{c}^{7}$ is an isometric immersion from a complete, simply connected and 6-dimensional nearly K\"ahler manifold into a standard space form with curvature $c$ then $M$ is a homogeneous Riemannian manifold and one of the following holds (up to homothty):\\
(i) $M$ is $S^{6}$, the space form is $\mathbb{R}^{7}$ and $f$ is equivalent to the standard embedding.\\
(ii) $M$ is $S^{3}\times S^{3}$, the space form is 7-sphere with $c=\sqrt{2}$ and $f$ is equivalent to the standard embedding $S^{3}\times S^{3}$ in $S^{7}_{\sqrt{2}}$.
\end{theorem}
\begin{proof}: Let $M$ be a 6-dimensional nearly
K\"ahler manifold such that the holonomy group of the intrinsic
Hermitian connection $\bar{\nabla}$ is strictly contained in
$SU_3$. As every maximal subgroup of $SU_3$ is conjugate to $U_2$,
this is equivalent to the reducibility of M. We fix such a
$U_2$ containing the holonomy group, which defines a
$\bar{\nabla}$-parallel complex line sub-bundle of $TM$, henceforth
denoted by $\nu$. Since  $\nabla_{X}Y=\bar{\nabla}_{X}Y$ where $X$
and $Y$ are complex linearly dependent, we deduce that $\nu$ is
totally geodesic for $\nabla$ (and in particular integrable). The
orthogonal complement of $\nu$ in $TM$ will be denoted by
$\mathcal{H}$  and the restriction of $\mathcal{H}$ to every
integral manifold $S$ of $\nu$ is identified with the normal
bundle of $S$ \cite{bel}. Therefore there exists a Riemannian
manifold $N$ and a Riemannian submersion with totally geodesic
fibers $\pi:M\longrightarrow N$ such that the tangent space of the
fiber through any point $x\in M$ is $\nu_{x}$ \cite{bel}.
Moreover, $N$ is an oriented self-dual, Einstein with
positive scalar curvature and $\pi:M\longrightarrow N$ is the
$S^2$-bundle over $N$ whose fiber over $x\in N$ consists of all complex structures
on $T_{x}N$ that are compatible with the metric and
orientation\cite{Rey}. We show that $M$ could not be a
hypersurface in space form. For every $x\in N$ there exist
open simply connected neighborhood  $U$ of $x$ such that $\pi^{-1}(U)\approx U\times
S^{2}$. Indeed this is an isometry because by pulling back the metric of $U\times S^{2}$on $\pi^{-1}(U)$ we have a 6-dimensional nearly K\"ahler manifold which is locally homogeneous by  theorem \ref{main theorem} and
corollary \ref{collorally}. By Butrulle theorem \cite{But} this metric coincides with the metric inherited from $M$.
Now, if $M$ is a hypersurface in the standard space form, restricting
the isometric immersion from $M$ into $\mathbb{Q}^{7}$ to
$\pi^{-1}(U)$ and composing it with the isometry coming
from $\pi^{-1}(U)\approx U\times S^{2}$ there
exist an isometric immersion from $U\times S^{2}$ into
$\mathbb{Q}^{7}$. If the space form is the Euclidean space then we have a contradiction with the fact that $N$ has positive scalar curvature. If the space form is a sphere
then again we get a contradiction with the fact that $N$ is self-dual or non-compact.
Therefore complete simply connected nearly K\"ahler hypersurface of
a space form could not have reduce holonomy of intrinsic Hermitian
connection and holonomy of intrinsic Hermitian connection must be
$SU_3$, by the main theorem of \cite{bel} this hypersurface
can not be neither $\mathbb{C}P^{3}$ nor $\mathbb{F}^3$. By Theorem
\ref{main theorem} this hypersurface is homogeneous and must be one
of $S^{6}$ or $S^{3}\times S^3$ and the result follows from
corollary\ref{collorally}.
\end{proof}
\paragraph*{}
For a $2n-$dimensional nearly K\"ahler manifold $M$ with strictly positive Ricci curvature
if $M$ is a hyperspace in Euclidean space then
$Ric(X_{i},X_{i})=\lambda_{i}\sum_{j\neq i}\lambda_{j}>0$ where
$\lambda_{i}$'s are principal curvatures and $X_{i}$'s are the corresponding
eigenvectors. In particular $\lambda_{i}$'s are all positive or
all negative. Thus the Gaussian curvature of the hypersurface, $\lambda_1...\lambda_{2n}$, is
strictly positive. Hence if $M$ is complete, the Guassian sphere
map $M\longrightarrow S^{2n}$ is a diffeomorphism \cite{Kob}. In particular $M$ is a
strictly convex hypersurface.
\paragraph*{}
If  $M$ is connected and simply connected the leaves of complex and invariant intrinsic Hermitian totally umbilic foliation on $M$ are homogeneous 6-nearly K\"ahler manifolds by our main theorem. When $M$ is a hypersurface in the Euclidean space all the six number of principle curvatures coincide.\\
When $n=4$, the remaining principle curvatures  $\lambda_{7},\lambda_{8}$ are not equal. Otherwise, if since the distribution $x\longmapsto\{X\in T_{x}M|AX=\lambda_{7} X\}$ is invariant under almost complex structure and torsion of intrinsic Hermitian connection (because $X_{7}=JX_{8}$) each leaf of this foliation must have dimension six which is impossible. Therefore every complete simply connected strictly 8-nearly K\"ahler hypersurface in the Euclidean space must be isometric with a product of a 6-dimensional nearly K\"ahler hypersurface in the Euclidean space (listed above theorem) with a compact, oriented 2-dimensional surface with positive Guassain curvature and non-equal principal curvatures.\\
For $n=5$ by similar argument one cane show that four remaining principle curvatures are mutually distinct.

\bibliographystyle{acm}

\end{document}